\documentclass[a4paper,10pt]{amsart}
\usepackage{amsthm, amssymb, amscd}
\usepackage{mathrsfs, verbatim}
\newtheorem{thm}{Theorem}
\newtheorem{prop}{Proposition}

\newtheorem{lm}[prop]{Lemma}

\newtheorem{definition}{Definition}

\def\OO{\mathcal{O}}

\def\N{\mathcal{N}}

\def\C{\mathbb{C}}

\def\PP{\mathbb{P}}
\def\Z{\mathbb{Z}}

\def\T{\mathcal{T}}

\begin{document}
\begin{abstract} Let  $V\cong \C^{d+1}$ be a complex vector space. If one identifies it with a space of binary forms of degree $d$, then one gets an action of $PGL(2)$ on any Grassmannian $Gr(e+1,V)$. We will produce some refined numerical invariants for such an action that stratify the Grassmannian into irreducible and rational invariant strata. Assuming $d\geq 3$, The numerical invariants so obtained are shown to correspond in a simple way with the set of possible splitting types of the restricted tangent bundles of degree $d$ rational curves $C\subset\PP^s$ with $s\leq d-1$. By means of the same techniques we produce explicit parametrizations for the varieties of rational curves with a given  splitting type of the restricted tangent bundle.  \end{abstract}
\title[Rational curves with given restricted tangent bundle]{PGL$(2)$ actions on Grassmannians and projective construction of rational curves with given restricted tangent bundle.}
\author{A. Alzati} 
\author{R. Re}
\address{Alberto Alzati,  Dipartimento di Matematica F.Enriques, Universit\`a di Milano,
via Saldini 50, 20133 Milano (Italy)}
\email{alberto.alzati@unimi.it}
\address{Riccardo Re, Dipartimento di Matematica e Informatica, Universit\`a di Catania, viale Andrea Doria 6, 95125 Catania (Italy)}
\email{riccardo@dmi.unict.it}
\noindent
\date{}
\maketitle
\section{Introduction}
A complex Grassmannian $Gr(e+1,V)$, with $V\cong \C^{d+1}$, has a natural class of isomorphic $PGL(2)$ actions on it. Indeed for any identification $V\cong S^{d}U$, with $U\cong \C^2$, one gets the irreducible representation of $\operatorname{Aut}(U)\cong GL(2)$ on $V$ by taking the symmetric powers $S^d(g)$ of automorphisms $g\in \operatorname{Aut}(U)$.  
In this article we study such a $PGL(2)$ action on $Gr(e+1,V)$. Precisely we will obtain a set of numerical invariants of points $[T]\in Gr(e+1,S^dU)$, corresponding to $(e+1)$-dimensional subspaces $T\subset S^dU$, that we call the {\em numerical type} of $T$. These numerical invariants induce a $PGL(2)$-invariant stratification of the Grassmannian that will turn out to have irreducible and rational strata. The classification of subspaces $T\subset S^dU$ according to their numerical type is the object of Theorem \ref{thm:main1}, proved in section 3. The basic idea behind this classification is to consider the $GL(2)$-invariant contraction map $\delta:T\otimes U^\ast\to S^{d-1}U$ and find some canonical form for this map. We will be able to decompose the map $\delta$ in a block form where the building blocks correspond to spaces  of two kinds. A space $T\subseteq S^dU$ is of the first kind if $\PP(T)$ is generated by its schematic intersection with the rational normal curve $C_d=\nu_d(\PP^1)\subset \PP(S^dU)$. 
The map $\delta$ for such a space may be considered as the analogue of a {\em semisimple}  $\C$-vector space endomorphism. A space $T\subset S^dT$ is of the second kind if it is generated by all the $r$-th derivatives $\partial_x^{r-i}\partial_y^i(g)$ of a form $g\in S^{d+r}U$, provided that these derivatives are linearly independent. The map $\delta$ for such a space may be seen as the analogue of a cyclic {\em nilpotent} $\C$-vector space endomorphism.

In section 4 we describe the subvariety $G_\tau$ of $Gr(e+1,S^dU)$ associated to a given numerical type $\tau$, showing that it is always irreducible and rational, as mentioned above. Moreover we give a formula for the dimension of these varieties. This is the content of Theorem \ref{thm:main2}, section 4. 

Finally, in section 5, we give an application of our results to the construction of degree $d$ rational curves $C\subset \PP^{d-e-1}$ with a given splitting type of the restricted tangent bundle. Precisely, one can obtain such a curve, up to projective equivalence, as a projection
of a rational normal curve $C_d\subset \PP(S^dU)$ from a vertex $\PP(T)$, with $T\subseteq S^dU$ and we will relate the splitting type of the restricted tangent bundle $\T_{\PP^{n-e-1}}|_C$ as a vector bundle on $\PP^1$ with the numerical type of $T$ studied in the previous sections. This is stated in Theorem \ref{thm:main3}, section 6 and in the concluding remarks of Section 7.

\section{Notations.}\label{sec:notations} Given a $\C$-vector space $V$, we denote $\PP(V)$ the projective space of $1$-dimensional subspaces of $V$. More generally we denote $Gr(e+1,V)$ or $Gr(e,\PP(V))$ the Grassmannian of $(e+1)$-dimensional subspaces of $V$, or equivalently, of the $e$-dimensional linear subspaces of $\PP(V)$. If $S\subseteq V$ is a $e+1$ dimensional subspace we will denote $[S]$ or $[\PP(S)]$ its associated point in $Gr(e,\PP(V))$. Accordingly, if $v\in V$ is a non-zero vector, we will denote $[v]\in \PP(V)$ its associated point.

Let $U\cong\C^2$ be a two dimensional vector space, $\PP^1=\PP(U)$ its associated projective line and let us fix a basis $x,y\in U$. Let $S^dU$ be the $d$-th symmetric product of $U$. We will denote $\nu_d:\PP(U)\to \PP(S^dU)$ the $d$-th Veronese map defined by $\nu_d(p)=[p^d]$ and $C_d=\nu_d(\PP(U))$ the rational normal curve given as the image of this map, i.e. the set of pure tensors in $S^dU$. For any $r\geq 1$ we denote $Sec^{r-1}C_d$ the closure of the set of $[\tau]\in\PP(S^dU)$ such that  $\tau=p^d_1+\cdots+ p_r^d$, for $[p_i]\in C_d$ distinct points, i.e. the $(r-1)$-th  secant variety of $C_d$. Let $g(x,y)=a_0x^d+\cdots+a_dy^d\in S^dU$ be a binary form of degree $d$ 
and
$\partial_x^{r-i}\partial_y^{i}g(x,y)=A_{i,0}x^{d-r}+\cdots+A_{i,d-r}y^{d-r}$.  Letting $g$ vary in $S^dU$ one may consider $A_{i,j}$ as elements of $S^dU^\ast=H^0(\OO_{\PP(S^dU)}(1))$ and one may define the matrix of $1$-forms $A_r=(A_{i,j})$ with indices $i\in\{0,\ldots,r\}$,  $j\in\{0,\ldots,d-r\}$. 
 In the following proposition we state the well known connection between this matrix and the secant variety $Sec^{r-1}C_d$.
\begin{prop}\label{prop:eqsec} The following facts hold.
\begin{enumerate}
\item One has $\dim Sec^{r-1}C_d=\min(2r-1,d)$ 
\item The $(r+1)\times (r+1)$ minors of the matrix $A_r$ generate the homogeneous ideal of $Sec^{r-1}C_d\subseteq\PP^d$.
\end{enumerate}
\end{prop}
A proof of these facts can be found for example in \cite{Geram}. 
An equivalent way to formulate the property (2) above is the following. Let us denote \begin{equation*}\Phi_r:\PP^d\dashrightarrow  \PP(\wedge^{r+1}S^dU^\ast)\end{equation*} the map defined by \begin{equation}\label{eq:Phie}\Phi_r(g)=[\partial_x^{r}g\wedge\partial_x^{r-1}\partial_yg\wedge\cdots\wedge \partial_y^rg].\end{equation} Then this map has base locus equal to $Sec^{r-1}C_d$.

\paragraph{} Let us consider a subspace $T\subseteq S^dU$, with $d\geq 2$. Observe that if $T$ does not contain pure tensors $p^d$ then $T$ is at least $2$-codimensional in $S^dU$, so $\dim T\leq d-1$.  
We introduce the contraction  map \begin{equation}\label{eq:defdelta}\delta:S^dU\otimes U^\ast\to S^{d-1}U\end{equation} and define $$\partial T=\delta(T\otimes U^\ast),$$ in particular the definition of $\partial T$ does not depend on the choice of the basis $x,y$. Note that one can identify $U^\ast=\langle\partial_x,\partial_y\rangle$ and the action $\delta$ with the action by derivation. One has also 
\begin{equation}\label{eq:deltaT}\partial T=\sum_{D\in U^\ast} DT.\end{equation}
From the definition of $\partial T=\delta(T\otimes U^\ast)$, one sees also that $\partial T=DT+ET$ with $D$ and $E$ any two linearly independent operators.
We also define a space $\partial^{-1}T\subset S^{d+1}U$ in the following way.
\begin{equation}\label{eq:delta-1T}\partial^{-1} T=\bigcap_{D\in U^\ast}D^{-1}T.\end{equation}
 Note that if $D,E\in U^\ast$ are linearly independent, then 
\begin{equation}\label{eq:delta-1formula}\partial^{-1} T=D^{-1}T\cap E^{-1}T.\end{equation}
Note that one always has 
\begin{equation}\label{eq:delta-1incl} \partial(\partial^{-1}T)\subseteq T\mbox{ and }T\subseteq \partial^{-1}(\partial T). \end{equation}
For $g\in S^{d+r}U$ we introduce the linear system \begin{equation}\label{eq:deltaetau}\partial^r(g)=\langle\partial_x^rg,\partial_x^{r-1}\partial_yg,\ldots,\partial_y^rg\rangle\subseteq S^dU.\end{equation} As a convention we set $\partial^r(g)=0$ if $r=-1$.
\section{The numerical type of a subspace $T\subseteq S^dU$.}
The object of this section is to prove the following general classification result of subspaces $T\subseteq S^dU$ according to a block decomposition of the linear map $\delta$.
\vskip2mm 
\paragraph{\em Notation.} Given a subspace $T\subseteq S^dU$, we denote $S=S_T$ the smallest subspace containing the schematic intersection
$\PP(T)\cap C_d$ as a subscheme. We set $a=\dim S-1=\dim \PP(S)$, with the convention that $\dim\emptyset=-1$. 
\begin{definition}\label{def:secant} We will say that a linear space $\PP(S)\subseteq\PP^d$ is $C_d$-generated  if
 $\PP(S)$ is generated by its schematic intersection with $C_d$. Setting  $a+1=\dim S$, we will also say in this case that $\PP(S)$ is $a+1$-secant to $C_d$.  We will say that a vector subspace $S\subseteq S^d U$ is $C_d$-generated if $\PP(S)$ is $C_d$-generated. \end{definition}
\begin{thm}\label{thm:main1} Let $T$ be a proper subspace $T\subseteq S^dU$. Let $S=S_T$ as defined above. Then $\dim \partial S=\dim S$. Moreover if we define $r = \dim\partial T - \dim(T)$ then $r\geq 0$ and either $r=0$ and in this case one has $T=S$, or $r\geq 1$ and there exist forms $ f_1,\ldots,f_r$, with $f_i\in \PP^{d+b_i}\setminus Sec^{b_i}C_{d+b_i}$ for $i=1,\ldots,r$, such that $T$ and $\partial T$ are the direct sums \[\begin{array}{l}T=S\oplus\partial^{b_1}(f_1)\oplus\cdots\oplus\partial^{b_r}(f_r)\\

\partial T=\partial S\oplus\partial^{b_1+1}(f_1)\oplus\cdots\oplus\partial^{b_r+1}(f_r).\end{array}\]
The $(r+1)$-uple $(a,b_1,\ldots,b_r)$ with $b_1\geq \cdots\geq b_r$ is uniquely determined from $T$. A space $T$ as above exists if and only if $a\geq -1$, $b_i\geq 0$ for all $i=1,\ldots,r$ and $a+1+\sum (b_i+2)\leq d.$ \end{thm}
\begin{definition}\label{def:type} We say that a subspace $T$ as in Theorem  \ref{thm:main1} has {\em numerical type} $(a,b_1,\ldots,b_r)$. \end{definition}
We divide the proof in steps. The first step shows how to algebraically determine the numerical type of $T$.
\vskip2mm
\paragraph{\bf Step 1.} In order to show the existence of a decomposition of $T$ as stated in Theorem \ref{thm:main1} we will apply the Kronecker normal form of a pencil of linear maps, see \cite{Gant}, chapter XII.
The Kronecker normal form of a pencil of linear maps $\Theta(\lambda,\mu)=\lambda f+\mu g:\C^n\to\C^m$ consists in a decomposition \begin{eqnarray*}\C^n\cong A\oplus B'_1\oplus\cdots\oplus B'_r\oplus C_1''\oplus\cdots\oplus C_s''\oplus \C^\alpha\\
\C^m\cong A'\oplus B_1''\oplus\cdots\oplus B_r''\oplus C_1'\oplus\cdots\oplus C_s'\oplus \C^\beta\end{eqnarray*} with $\dim A=\dim A'$, $\dim B_i''=\dim B'_i+1$ and $\dim C''_j=\dim C_j'+1$ for any $i,j$ and a decomposition of the pencil  $\Theta(\lambda,\mu)$ into a direct sum of pencils of linear maps \begin{eqnarray*}
\Theta(\lambda,\mu)|_{\C^\alpha}=0\\
\Theta_A(\lambda,\mu):A\to A'\\
\Theta_{B_i}(\lambda,\mu):B_i'\to B_i''\\
\Theta_{C_j}(\lambda,\mu):C_j''\to C_j' \end{eqnarray*} such that $\Theta_A(\lambda,\mu)$ is a isomorphism for general $(\lambda:\mu)\in \PP^1$, 
$\Theta_{B_i}(\lambda,\mu)$ are represented by matrices of type $(\dim B'_i+1)\times\dim B'_i$ of the form \begin{equation}\label{eq:Bblockmatrix}M=\left(\begin{array}{llll}\lambda&0&\cdots&0\\
\mu&\lambda&\ddots&0\\
0&\ddots& \ddots&0\\
0&&\mu&\lambda\\
0& &0&\mu\end{array}\right)\end{equation}  and $\Theta_{C_j}(\lambda,\mu)$ are represented by  matrices of the form $M^t$. Note that $\Theta_{B_i}(\lambda,\mu)$ are injective for any $(\lambda:\mu)\in \PP^1$ and $\sum \operatorname{Im}\Theta_{B_i}(\lambda,\mu)=B_i''$. Moreover $\Theta_{C_j}(\lambda,\mu)$ have $1$-dimensional kernel for any $(\lambda:\mu)\in \PP^1$.

We will consider the Kronecker normal form of the pencil $$\lambda\partial_x+\mu\partial_y: T\to \partial T.$$
We can immediately exclude the existence of the $\C^\alpha$ and $\C^\beta$ summands, since the subspace $\C^\alpha$ is contained in the kernel of any $D=\lambda\partial_x+\mu\partial_y$ and hence it is zero, and $\C^\beta$ is embedded in the cokernel of any $D$, but we know that the spaces  $DT$ generate $\partial T$. 

Moreover we can exclude the presence of the``$C$" summands in the case when $T$ is a proper subspace of $S^dU$, because of the following proposition.
\begin{prop}\label{prop:dimT-1} For any non zero $T$ one has $\dim\partial T\geq \dim T-1$. The only subspace $T\subseteq S^dU$ such that $\dim\partial T=\dim T-1$ is $T=S^dU$. \end{prop}
\begin{proof} Note that for any non zero $D\in U$ one can write $D$ as the derivation $\lambda \partial_x+\mu\partial_y$ and one has $\ker D=(p^d)$ with $p=-\mu x+\lambda y\in\PP(U)$, i.e. with $(D)=p^\perp\subseteq U^\ast$. One has $\partial T\supset DT$ and $\ker D|_T$ at most $1$-dimensional, so $\dim \partial T\geq \dim T-1$. 

If $\dim \partial T=\dim T-1$, then for any $D\in U^\ast$ one has $\ker D|_T\not=0$, hence $p^d\in T$ for all $p\in U$, hence $T=S^dU$.
\end{proof}
In the following we will characterize the summands of $T$ of type $``A"$ and $``B"$.
\paragraph{\bf Step 2.} We classify all non zero subspaces $T\subseteq S^dU$ of type $``A"$, i.e. such that $\dim \partial T=\dim T$. 
\begin{prop}\label{prop:secant}  Let $h$ and $k$ be positive integers. The following facts are equivalent for a proper $(k-1)$-dimensional linear subspace space $\PP(S)\subset \PP(S^{h+k}U)$.
\begin{enumerate} \item $\PP(S)$ is a $k$-secant space to $C_{h+k}$.
\item  $S=\{f\in S^{h+k}U\ |\ H(f)=0\}$ for some non zero $H\in S^kU^\ast$.
\item $\dim\partial S=\dim S$.
 \end{enumerate}\end{prop}
\begin{proof}($1\Rightarrow 2$) If $\PP(S)$ is $k$-secant to a rational normal curve then it is well known that $\dim S=k$ and that $S$ is a limit of secant spaces of the form $S'=\langle p_1^{h+k},\ldots,p_{k}^{h+k}\rangle$ with $[p_i]$ distinct points on $\PP(U)$. Writing $p_i=\lambda_ix+\mu_iy$ let us set $D_i=-\mu_i\partial_x+\lambda_i\partial_y\in U^\ast$ for any $i=1,\ldots,k$. Then it is well known that $S'$ is the annihilator in $S^{h+k}U$ of $H=D_1\ldots D_k$, by the apolarity lemma, see \cite{Geram}. Alternatively one may note that $H:S^{h+k}U\to S^hU$ is surjective and it has a $k$ dimensional kernel containing $S'$ and hence equal to $S'$. Passing to the limit we see that also $S=\{f\in S^{h+k}U\ |\ H(f)=0\}$ for some non zero $H\in S^kU^\ast$. 
\vskip1mm
\paragraph{($2\Rightarrow 3$)} Let $S=\{f\in S^{h+k}U\ |\ H(f)=0\}$ for a given non zero $H\in S^kU^\ast$. Note that $\partial S\subset \{f\in S^{h-1+k}U\ |\ H(f)=0\}$  and this latter space has dimension equal to $k=\dim S$. Since $S\not=S^{h+k}U$ we have that $\dim \partial S=\dim S$. 
\vskip1mm
\paragraph{($3\Rightarrow 2$)}If $\dim \partial S=\dim S$ then there exists some $L\in U^\ast$ with a non zero kernel on $S$. Indeed the maps $L:S\to \partial S$, as $L$ varies in $U^\ast$ are a non-constant family of maps between vector spaces of the same dimension, and therefore one of them must have zero determinant. Take $S'=LS$, then $\dim S'=\dim S-1$. Note also that the kernel of $L$ in $S$ must have the form $q^{h+k}$ for some $q\in U$, hence $q^{h+k-1}\in\partial S$ and it is still annihilated by $L$. Moreover $\partial S'=L\partial S$, so we see that $\dim \partial S'\leq \dim \partial S-1=\dim S'$. Indeed one has $\dim \partial S'=\dim S'$, since the only case with the strict inequality above would be $S'=S^{h+k-1}U$, by Proposition \ref{prop:dimT-1}, and this is excluded by $\dim S'=\dim S-1<\dim S^{h+k}U-1=\dim S^{h+k-1}U$. Then by induction there exists some $H'\in S^{k-1}U^\ast$ with $0=H'S'=(H'L)S$, so the conclusion holds for $H=H'L$. 
\vskip1mm
\paragraph{($2\Rightarrow 1$)} Let $H=D_1\ldots D_k$ a product of pairwise non proportional operators. Let us write $D_i=-\mu_i\partial_x+\lambda_i\partial_y$ and $p_i=\lambda_ix+\mu_iy$ for any $i=1,\ldots,k$. Then one sees easily that
$$\langle p_1^{h+k},\ldots,p_k^{h+k}\rangle=\{f\in S^{h+k}U\ |\ H(f)=0\}.$$ 
Now an arbitrary $H\in S^kU^\ast$ is limit of operators as above and the space $S$ defined as $S=\{f\in S^{h+k}U\ |\ H(f)=0\}$ is a limit of spaces as above, since it has the same dimension $k=\dim \ker (H:S^{h+k}U\to S^{h}U)$ and therefore it defines an element of the same Grassmannian containing the spaces introduced above. Hence $S$ is $k$-secant, according to Definition \ref{def:secant}. 
\end{proof}
 \paragraph{\bf Step 3.} Now we will study the spaces $T$ of type $``B"$ i.e. such that $\dim \partial T=\dim T+1$ and $\PP(T)\cap C_d=\emptyset$. 

\begin{prop}\label{thm:derivsys} Let $T\subseteq S^dU$ be a subspace and set $e=\dim T-1$. Then $T$ does not contain pure tensors $p^d$ and is such that $\dim \partial T=\dim T+1$ if and only if there exists some $g\in S^{d+e}U$ such that $T=\partial^e(g)$  and $[g]\not\in Sec^{e}C_{d+e}$. \end{prop} 
\begin{proof} Since $T$ does not contain pure tensors one has $D:T\to \partial T$ injective for any non zero $D\in U^\ast$, in particular $\dim DT=\dim T$ for any $D\not=0$. If $DT$ were independent of $D$, one would have $DT=\partial T$ for any $D\not=0$, contradicting $\dim \partial T>\dim T$. Then $DT$ varies in the projective space of $1$-codimensional spaces of $\partial T$ and since $\dim \partial T=\dim T+1$ we have $\partial T=\bigcup_D DT$.
\paragraph{$(\Rightarrow)$} We  first show that $T$ has necessarily the form $\partial^e(g)$. Let us consider the subspaces $\partial_xT\subseteq\partial T$ and $\partial_yT\subseteq\partial T$ and let us denote $S=\partial_xT\cap\partial_yT$ their intersection. Let us also denote $T_x=\partial_x^{-1}(S)\cap T$ and $T_y=\partial_y^{-1}(S)\cap T$. Since $T$ does not contain pure tensors one has $\dim\partial_xT=\dim\partial_yT=\dim T$ and $\dim S=2\dim T-\dim\partial T=\dim T-1$. Hence one has also $\dim T_x=\dim T_y=\dim T-1$ and there are two alternatives: either $T_x\not=T_y$ and therefore $T=T_x+T_y$, or
$T_x=T_y$. Let us show that this last alternative does not occur. Indeed if $T_x=T_y=S'$, then $\delta: S'\otimes U^\ast\to S$ is well defined and, since $\dim S'=\dim S$, it can be described by a pencil of square matrices. As a consequence there exists some non zero $D\in U^\ast$ such that $D:S'\to S$ is not injective, so there must exist some $p^d\in S'$, which is excluded by the assumptions. Now we are left with $T=T_x+T_y$.  

Consider $T_1=\partial_y^{-1}T_x$ and $T_2= \partial_x^{-1}T_y$. These are subspaces of $S^{d+1}U$ of dimension equal to $\dim T$ and such that $\partial_x\partial_y T_1=\partial_x\partial_yT_2=S$, so also $\partial_x\partial_y(T_1+T_2)=S$. More precisely, we have $T_1+T_2=(\partial_x\partial_y)^{-1}(S)$, since $\ker (\partial_x\partial_y)=\langle x^{d+1},y^{d+1}\rangle\subseteq T_1+T_2$. We have therefore $\dim(T_1+T_2)=\dim S+2=\dim T+1$, hence $\dim (T_1\cap T_2)=\dim T-1$.  
Note that $T_1\cap T_2$ does not contain $x^{d+1}$ or $y^{d+1}$, indeed if, for example, $x^{d+1}\in T_1\cap T_2$ one would have $(d+1)x^d=\partial_x(x^{d+1})\in T_y\subseteq T$, which is impossible. Hence $\partial_x$ and $\partial_y$ are injective on $T_1\cap T_2$. Since $\partial_y (T_1\cap T_2)\subseteq T_x$ and $\partial_x (T_1\cap T_2)\subseteq T_y$, by the equality of dimensions, we have $\partial_y (T_1\cap T_2)= T_x$ and $\partial_x (T_1\cap T_2)=T_y$, so we have found $$\partial(T_1\cap T_2)=T,\mbox{ with } \dim (T_1\cap T_2)=\dim T-1,$$
and moreover $T_1\cap T_2$ not containing any pure tensor $p^{d+1}$, otherwise for a general $D\in U^\ast$ one would have $0\not=Dp^{d+1}=\lambda p^d\in T$.

So, assuming $T$ not containing pure tensors and $\dim\partial T=\dim T+1$, we have found a space $T^{(1)}$ not containing pure tensors, such that $\partial T^{(1)}=T$ and such that  $\dim\partial T^{(1)}=\dim T^{(1)}+1$. 

We can iterate the whole procedure finding spaces $T=T^{(0)}, T^{(1)},\ldots,T^{(e)}$ of dimension $\dim T^{(i)}=\dim T-i$ and such that $\partial T^{(i+1)}=T^{(i)}$ until we obtain $\dim T^{(e)}=1$. So we find $T^{(e)}=(g)\subseteq S^{d+e}$ and $T=\partial^e(g)$. 
\vskip2mm
\noindent
$(\Leftarrow)$ Now we examine the condition under which, for $T=\partial^e(g)$, one has $\dim\partial T=\dim\partial^{e+1}(g)=e+2$, and hence, as we will see, also $\dim T=e+1$. It is clear that $\dim\partial^{e+1}(g)=e+2$ if and only if the partial derivatives $\partial_x^{e+1-i}\partial_y^{i}g$ are independent for $i=0,\ldots,e+1$. This means that the $(e+2)\times d$ matrix $(A_{i,j})$ as in Section \ref{sec:notations} has rank $e+2$. By Proposition \ref{prop:eqsec} this is equivalent to $[g]\not\in Sec^eC_{d+e}$. Note that if $[g]\not\in Sec^eC_{d+e}$ then a fortiori $[g]\not\in Sec^{e-1}C_{d+e}$ and hence $\dim T=e+1$. If there were a point $p^d\in\PP(T)\cap C_d$, then we could find a basis $p,q$ of $U$ so that $\partial_qp=0$ and, since $p^d=Dg\in\partial^e(g)$, we would have also $\partial_qDg=0,$ so $\dim\partial^{e+1}(g)<e+2$, a contradiction. So, automatically, if $[g]\not\in Sec^eC_{d+e}$, then $T=\partial^e(g)$ does not contain pure tensors.
\end{proof}

The result of Proposition \ref{thm:derivsys} can be rephrased by saying that the subvariety of the Grassmannian $Gr(e,\PP(S^dU))$ of subspaces $\PP(T)\subseteq\PP(S^dU)$ such that $\dim \partial T=\dim T+1$ and $\PP(T)\cap C_d=\emptyset$ is the image $\Phi_e(\PP^{d+e}\setminus Sec^{e}C_{d+e})$ by means of the rational map
\begin{equation}\label{eq:Phi}\Phi_e: \PP^{d+e}\dashrightarrow  Gr(e,\PP^d)\subseteq \PP(\wedge^{e+1}S^dU)\end{equation} defined by $\Phi_e(g)=[\partial_x^eg\wedge\cdots\wedge \partial_y^eg]$. 
 \paragraph{\bf Step 5. Uniqueness of the $``A"$ summand.} We need to identify the $``A"$ summand of the Kronecker decomposition with the $C_d$ generated part $S_T$ of $T$, as stated in Theorem \ref{thm:main1}. Of course the $``A"$ summand is contained in $S_T$, since $\dim A=\dim\partial A$ and hence it is a secant space by Proposition \ref{prop:secant}. Assume that $A\subsetneq S_T$. Since $T=A\oplus\partial^{b_1}(f_1)\oplus\cdots\oplus\partial^{b_r}(f_r)$, we could write $S_T=A\oplus W$ with $W\subseteq\partial^{b_1}(f_1)\oplus\cdots\oplus\partial^{b_r}(f_r)$ and $\partial S_T=\partial A\oplus \partial W$. By construction of $W$ we have $D|_W:W\to \partial W$ injective for any $D\not=0$. This implies that $\dim \partial W\geq\dim W+1$, otherwise $W$ would be one of the spaces characterized in Proposition \ref{prop:dimT-1} or Proposition \ref{prop:secant}. This implies $\dim \partial S_T\geq \dim S_T+1$, which is impossible since $S_T$ is a secant space by construction  and then by Proposition \ref{prop:secant} one has $\dim \partial S_T=\dim S_T$.
\paragraph{\bf Step 6. Uniqueness of the numerical type.} The uniqueness of the numerical type $(a,b_1,\ldots,b_r)$ is a consequence of the uniqueness of the set of block sizes in the Kronecker normal form used above. 
 \paragraph{\bf Step 7. Existence. } Assume that $a+1+\sum (b_i+2)\leq d$. Then we construct a space $T$ of type $(a,b_1,\ldots,b_r)$ as follows.
 We consider a subspace $T\subset S^dU$ generated by a set of monomials arranged into $r+1$ separated sequences each made of monomials with consecutive $x$-degrees. A possible choice for such sequences is the following 
 \[\begin{array}{l} A=x^d,x^{d-1}y,\ldots,x^{d-a}y^a\\
 B_1=x^{d-a-2}y^{a+2},\ldots,x^{d-a-2-b_1}y^{a+2+b_1}\\
 \cdots\\
 B_r=x^{d-a-2-\sum_{i=1}^{r-1} (b_i+2)}y^{a+2+\sum_{i=1}^{r-1} (b_i+2)},\ldots,x^{d-a-\sum_{i=1}^{r} (b_i+2)}y^{a+\sum_{i=1}^{r} (b_i+2)}.\end{array}\]
 As we said, we take $T=\langle A,B_1,\ldots,B_r\rangle$. Note that $S=\langle A\rangle$ is a $C_d$-generated space, precisely it is the space generated by the $0$-dimensional subscheme of $C_d$ supported on $p=[x^d]\in C_d$ and of degree $a+1$. Moreover, setting 
 $$f_1=x^{d-a-2}y^{a+b_1+2}\mbox{ and } f_i=x^{d-a-2-\sum_{j=1}^{i-1} (b_j+2)}y^{a+\sum_{j=1}^{i} (b_j+2)} \mbox{ for } i=2,\dots,r,$$  one can see that $\langle B_i\rangle=\partial^{b_i}(f_i)$. Hence one can write $$T= S\oplus \partial^{b_1}(f_1)\oplus\cdots\oplus\partial^{b_r}(f_r).$$ To show that the type of $T$ is $(a,b_1,\cdots,b_r)$ it remains only to show that $$\partial T=\partial S\oplus \partial^{b_1+1}(f_1)\oplus\cdots\oplus\partial^{b_r+1}(f_r).$$ This is a straightforward calculation. Actually in this example one can show that  $\partial(T)$ is generated by the consecutive monomials $$x^{d-1},\ldots,x^{d-1-a-\sum_{i=1}^{r} (b_i+2)}y^{a+\sum_{i=1}^{r} (b_i+2)}$$ and so its dimension is $a+1+\sum_{i=1}^{r} (b_i+2)$ and hence $\dim\partial T=\dim T+r$, which forces the type of $T$ to be exactly $(a,b_1,\ldots,b_r)$. We omit the computational details.

\section{Varieties of subspaces $T\subseteq S^dU$ of given type $(a,b_1,\ldots,b_h)$}
We start with a study of some properties of the operator $\partial^{-1}$ defined by formula (\ref{eq:delta-1T}).
\begin{prop}\label{prop:d-1} Let $T\subseteq S^dU$ be a subspace and hence $\partial^{-1}T\subseteq S^{d+1}U$ and $\partial T\subseteq S^{d-1}U$.  The following properties hold for the operator $T\mapsto \partial^{-1}T$.
\begin{enumerate}
\item[\it(a)] If $T=(f)\not=0$ then $\partial^{-1}T=0$ if $[f]\not\in C_d$, otherwise $T=(p^d)$ and $\partial^{-1}T=(p^{d-1})$.
\item[\it(b)] If $T$ is $C_d$-generated, then $\dim \partial^{-1}T=\dim T$ and $\partial(\partial^{-1}T)=T$. In particular if $T=\langle p_0^d,\ldots,p_e^d\rangle$, then $\partial^{-1}T=\langle p_0^{d+1},\ldots,p_e^{d+1}\rangle$.
\item[\it(c)] If $T=\partial^b(f)$ with $b>0$ and $[f]\not\in Sec^bC_{d+b}$, then $\partial^{-1}T=\partial^{b-1}(f)$ and $\partial(\partial^{-1}T)=T$.
\item[\it(d)] If $T\subseteq S^d$ is a proper subspace of the form $T=A\oplus B$ and $\partial T=\partial A\oplus\partial B$, then $\partial^{-1}T=\partial^{-1}A\oplus\partial^{-1}B$.
\item[\it(e)] If $T$ is a space as in theorem \ref{thm:main1}, that is $T=S\oplus\partial^{b_1}(f_1)\oplus\cdots\oplus\partial^{b_r}(f_r)$ and $\partial T=\partial S\oplus\partial^{b_1+1}(f_1)\oplus\cdots\oplus\partial^{b_r+1}(f_r)$, with $S=S_T$ and $[f_i]\not\in Sec^{b_i}C_{d+b_i}$ for $i=1,\ldots,r$, then
$\partial^{-1}T=\partial^{-1}S\oplus\bigoplus_{b_i\geq 1}\partial^{b_i-1}(f_i).$
\end{enumerate}\end{prop}
\begin{proof} 
{\it(a)}. It is clear that $\partial^{-1}(p^d)=(p^{d+1})$. Conversely, if $T=(f)$ then for any $g\in \partial^{-1}T$ one has $\partial_xg$  and $\partial_yg$ proportional to $f$, so they are proportional and this means that $g=p^{d+1}$ for some $p$, hence $f=p^d$. 

\noindent
{\it(b)}. If $T$ is $C_d$-generated, then by Proposition \ref{prop:secant} one has $T=\operatorname{Ann}(D)\subset S^dU$ for some $D\in S^{a+1}U^\ast$, with $a+1=\dim T$. We set $T_1=\operatorname{Ann}(D)\subset S^{d-1}U$ and observe that $\partial T_1=T$, since one has $\partial T_1\subseteq T$ and their dimensions are equal. Hence $T_1\subseteq \partial^{-1}T$. To show that $T_1=\partial^{-1}T$ we observe that otherwise one would have $\dim T= \dim \partial(\partial^{-1}T)=\dim \partial T_1=\dim T_1<\dim\partial^{-1}T$, which would imply $\partial^{-1}T=S^{d+1}U$ by Proposition \ref{prop:dimT-1}, a contradiction.

\noindent
{\it(c)}. Given $T=\partial^b(f)$ with $[f]\not\in Sec^{b}C_{d+b}$, we set $T_1=\partial^{b-1}(f)\subseteq \partial^{-1}T$.  Note that $\partial T_1=T$. Assume that $T_1\subsetneq \partial^{-1}T$. We have in any case $T=\partial T_1\subseteq \partial(\partial^{-1}T)\subseteq T$, so the inclusions are all equalities, hence  $b+1=\dim T= \dim\partial(\partial^{-1}T)=\dim \partial T_1=\dim T_1+1\leq \dim\partial^{-1}T$, hence $\partial^{-1}T$ would be a $C_{d+1}$-generated space by Proposition \ref{prop:secant}. In particular $\partial^{-1}T$ would contain pure tensors $p^{d+1}$, so  $T$ would contain $p^d$, which is excluded by Proposition \ref{thm:derivsys}.

\noindent
{\it(d)}. First of all observe that $\partial^{-1}A\cap\partial^{-1}B=0$, since any $f\in\partial^{-1}A\cap\partial^{-1}B$ would have the property that $Df\in A\cap B=0$ and hence $Df=0$ for any $D$, hence $f=0$. So the sum $\partial^{-1}A+\partial^{-1}B$ is direct and moreover it is trivially contained in $\partial^{-1}(A\oplus B)$. To show the converse inclusion we consider two general operators $D,E\in U^\ast$ and the representation $$\partial^{-1}(A\oplus B)=(D^{-1}A+ D^{-1}B)\cap (E^{-1}A+ E^{-1}B).$$ 
We consider an arbitrary element $f\in \partial^{-1}(A\oplus B)$, which we can represent in two ways $$f=g_A+g_B=h_A+h_B,$$ relatively to the fixed operators $D$ and $E$,  with $Dg_A\in A$, $Dg_B\in B$, $Eh_A\in A$, $Eh_B\in B$.  Then we have $g_A-h_A=h_B-g_B$ and applying $ED$ to both sides we find $$EDg_A-DEh_A=DEh_B-EDg_B\in \partial A\cap\partial B=0,$$ hence  $$DE(g_A-h_A)=DE(h_B-g_B)=0.$$ Since $\ker DE=\langle p_D^{d+1},p_E^{d+1}\rangle$, we find $$g_A-h_A=h_B-g_B=\lambda p_D^{d+1}+\mu p_E^{d+1},$$

In particular, applying $D$ we find $Dh_A\in A+\langle p_E^{d}\rangle$ and  $Dh_B\in B+\langle p_E^{d}\rangle ,$ hence $$h_A\in D^{-1}(A+\langle p_E^{d}\rangle),\quad h_B\in D^{-1}(B+\langle p_E^{d}\rangle).$$  Since $D$ is surjective, one has $D^{-1}(A+\langle p_E^{d}\rangle)=D^{-1}A+D^{-1}\langle p_E^{d}\rangle=D^{-1}A+\langle p_D^{d+1}\rangle+\langle p_E^{d+1}\rangle=D^{-1}A+\langle p_E^{d+1}\rangle$. Similarly one has $D^{-1}(B+\langle p_E^{d}\rangle)=D^{-1}B+\langle p_E^{d+1}\rangle$. Now, since $h_A\in D^{-1}A+\langle p_E^{d+1}\rangle$, we can write $h_A=g+\lambda p_E^{d+1}$ with $\lambda\in\C$ and $g\in D^{-1}A$ and applying $E$ we see $Eg=Eh_A\in A$, so $g\in D^{-1}A\cap E^{-1}A=\partial^{-1}A$. We have found $h_A\in \partial^{-1}A+\langle p_E^{d+1}\rangle$.
Similarly, we have $h_B\in \partial^{-1}B+\langle p_E^{d+1}\rangle$. Since we have $f=h_A+h_B$, and letting $E$ vary, we see that
  $$f\in \bigcap_E (\partial^{-1}A+\partial^{-1}B+\langle p_E^{d+1}\rangle )=\partial^{-1}A+\partial^{-1}B.$$ The last equality holds because the spaces $\partial^{-1}A+\partial^{-1}B+\langle p_E^{d+1}\rangle $ for varying $E$ are not all equal, otherwise they would be all equal to $S^{d+1}U$ and hence applying $E$ one would find $T=A\oplus B=S^dU$, against the assumptions.

\noindent
{\it(e)}. One proves the statement by iterated applications of  {\it(d)} to the decomposition $T=S\oplus\partial^{b_1}(f_1)\oplus\cdots\oplus\partial^{b_r}(f_r)$ and applying {\it(a)}, {\it(b)} and {\it(c)} to the various summands.

\end{proof}
\subsection{The injectivity locus of $\Phi_e$}
Now we discuss the question of the injectivity of $\Phi_e$ defined by the equation (\ref{eq:Phie}). We have seen that $\Phi_e|_{Sec^eC_{d+e}}$ is certainly not injective, indeed it has general fibers equal to the secant $e$-dimensional spaces to $C_{d+e}$. 
We have the following result. 
\begin{prop}\label{thm:injlocus} The map $\Phi_e:\PP^{d+e}\dashrightarrow Gr(e,\PP^{d})$ is a birational embedding off the closed subset $Sec^eC_{d+1}$.  \end{prop}
\begin{proof} First we show that $\Phi_e^{-1}(\Phi_e(f))=[f]$ for any $[f]\in\PP^{d+e}\setminus Sec^bC_{d+e}$. Indeed for $T=\partial^e(f)$ and by iterated applications of Proposition \ref{prop:d-1}, one has $\partial^{-e}T=(f)$. So $\Phi_e$ is injective on $\PP^{d+e}\setminus Sec^eC_{d+e}$. Indeed one can even show that the differential of $\Phi_e$ at a point $[g]\in \PP^{d+e}\setminus Sec^bC_{d+e}$ is injective, so that $\Phi_e$ is an embedding. A tangent vector at $[g]$ may be represented as  a class $\bar{h}\in S^{d+e}U/\langle g\rangle$ and the differential can be calculated as follows:
$$d\Phi_e(\bar{h})= \partial_x^e h\wedge\cdots\wedge \partial_y^eg+\cdots+\partial_x^eg\wedge\cdots\wedge\partial_y^e h\mbox{ mod }\Phi_e([g]).$$
Assume that $d\Phi_e(\bar{h})=0$. Hence up to changing the representative $h$ by adding a suitable multiple of $g$, we may assume 
$$\partial_x^e h\wedge\cdots\wedge \partial_y^eg+\cdots+\partial_x^eg\wedge\cdots\wedge\partial_y^e h=0.$$
Now we take the wedge product of the relation above with $\partial_x^eg,\ldots,\partial_y^eg$ respectively, and after easy manipulations we find $$\partial_x^{e-i}\partial_y^i h\wedge \partial_x^eg,\wedge\cdots\wedge\partial_y^eg=0$$ for any $i=0,\ldots,e$. This means $\partial^e(h)\subseteq \partial^e(g)$. Now we find $\langle h\rangle\subseteq\partial^{-e}\partial^e(h)\subseteq \partial^{-e}\partial^e(g)=\langle g\rangle$, the last equality holding by the assumption $[g]\in\PP^{d+e}\setminus Sec^eC_{d+e}$, as discussed above. Hence $\bar{h}=0$, as claimed.\end{proof}
\paragraph{\bf Remark} One can also prove that the birational map $\Phi_e$ contracts any $e$-dimensional linear space $\PP(S)\subset\PP(S^{d+e}U)$ that is a $e+1$ secant space, to the point $[\partial^eS]\in Gr(e,\PP^d)$. We omit the details.
\subsection{The variety of representations of a given space $T$}
Let $T\subseteq S^dU$ be a proper subspace with numerical type $(a,b_1,\ldots,b_r)$ and let $S$ the $C_d$-generated part of $T$, with $\dim S=a+1$, $a\geq -1$. We define the subvariety $V_T\subseteq \PP^{d+b_1}\times\cdots\times \PP^{d+b_r}$ whose points are the $r$-tuples $([f_1],\ldots,[f_r])$
representing $T$ as \begin{equation}\label{eq:Trepr}T=S\oplus\partial^{b_1}(f_1)\oplus\cdots\oplus\partial^{b_r}(f_r),\end{equation} in such a way that $\partial T=\partial S\oplus\partial^{b_1+1}(f_1)\oplus\cdots\oplus\partial^{b_r+1}(f_r)$, as in Definition \ref{def:type} and Theorem \ref{thm:main1}. The object of this subsection is to compute the dimension of $V_T$. Observe that if $([f_1],\ldots,[f_r])$ correspond to a fixed representation of $T$ and $([f_1'],\ldots,[f_r'])$ is a small deformation of $([f_1],\ldots,[f_r])$ subject only to the conditions $\partial^{b_i}(f_i')\subseteq T$ for any $i=1,\ldots r$, then $([f_1'],\ldots,[f_r'])$ will also provide a representation for $T$, since the requirements that  $S\oplus\partial^{b_1}(f_1')\oplus\cdots\oplus\partial^{b_r}(f_r')$ and $\partial S\oplus\partial^{b_1+1}(f_1')\oplus\cdots\oplus\partial^{b_r+1}(f_r')$ are direct sums are open conditions.
Let us define, for any integer $b\geq 0$, the linear space $$\PP(F_{T,b})=\{[f]\in \PP^{d+b}\ |\ \partial^{b}(f)\subseteq T\}.$$
So we see that $V_T$ contains a dense open set of $\PP(F_{T,b_1})\times \cdots\times \PP(F_{T,b_r})$.
Hence $$\dim V_T=\dim \PP(F_{T,b_1})+\ldots+\dim \PP(F_{T,b_r}).$$
\paragraph{\bf Computation of $\dim F_{T,b}$} We have $f\in F_{T,b}$ if and only if $f\in \partial^{-b}T$, so we have the identification
$$F_{T,b}=\partial^{-b}T=\partial^{-b}S\oplus\bigoplus_{b_i\geq b}\partial^{b_i-b}(f_i),$$
by iterated applications of Proposition \ref{prop:d-1}, (e). Hence, if $b\leq b_i$ for some $i$, we find 
$$\dim \PP (F_{T,b})=\dim S+\sum_{b_i\geq b}(b_i-b+1)-1=a+\sum_{b_i\geq b}(b_i-b+1).$$
\paragraph{\bf Computation of $\dim V_T$} We may assume $b_1\geq \cdots\geq b_r\geq 0$. Applying the formula obtained above to $b=b_1,\ldots,b_r$, we obtain \begin{equation}\label{eq:dimVT}\dim V_T
=ra+\sum_{i=1}^r\sum_{b_j\geq b_i}(b_j-b_i+1).\end{equation}

\subsection{Dimensions of the varieties parametrizing  spaces $T$ of a given type}
Finally, given a decomposition type $\tau=(a,b_1,\ldots,b_r)$ as in Theorem \ref{thm:main1} and Definition \ref{def:type}, setting $e+1=\dim T=a+1+\sum b_i+r$ we define the subvariety of the Grassmannian $Gr(e,\PP^d)$ of spaces $\PP(T)$ of type $\tau$:
$$G_\tau=\{[T]\in Gr(e,\PP^d)\ |\ T \mbox{ of type }\tau\}.$$
We know that $G_\tau$ is parameterized by the $(r+1)$-tuples $(S,f_1,\ldots,f_r)$ with $\PP(S)$ a $(a+1)$-secant $a$-plane to $C_d$ and $[f_i]\in \PP^{d+b_i}\setminus Sec^{b_i}C_{d+b_i}$ for $i=1,\ldots,r$. The fibers of this parametrization are the varieties $V_T$, which are open sets in products of projective spaces. Taking into account Proposition \ref{thm:injlocus}, we see that
$$\dim G_\tau=a+1+rd+\sum b_i-\dim V_T,$$ with $\dim V_T$ given by formula (\ref{eq:dimVT}).
Then we can state the following theorem.
\begin{thm}\label{thm:main2} The subvarieties $G_\tau\subset Gr(e, \PP^d)$ parametrizing spaces $T\subset S^{d}U$ with $\dim T=e+1$ and numerical type of the form $\tau=(a,b_1,\ldots,b_r)$ are irreducible rational quasi-projective varieties of dimension
\begin{equation}\label{eq:dimGtau}\dim G_\tau=e+1+r(d-a-1)-\sum_{i=1}^r\sum_{b_j\geq b_i}(b_j-b_i+1).\end{equation} \end{thm}
\begin{proof} Let us denote $C^{(a)}$ the $a$-fold symmetric product of $C$, whose points are identified with degree $d$ effective divisors $p_1+\cdots+p_a$ of $C$. The irreducibility of $G_\tau$ is a consequence of the existence of a dominant rational map $\Psi:C_d^{(a)}\times\PP^{b_1}\times\cdots\times\PP^{b_r}\dashrightarrow G_\tau$ defined by $$(p_0+\cdots+p_a, [f_1],\ldots,[f_r])\mapsto [\langle p_0^d,\ldots,p_a^d\rangle\oplus\partial^{b_1}(f_1)\oplus\cdots\oplus\partial^{b_r}(f_r)].$$ The formula for the dimension of $G_\tau$ is a consequence of the calculations made above. Hence it remains to prove only the rationality of $G_\tau$. Note that the existence of the map above immediately implies the {\em unirationality} of $G_\tau$. Recall that the fibers of the map above have the form $\{D\}\times V_T$ with $D\in C_d^{(a)}$ and $V_T$ as denoted in the discussion above.
Assuming that $r=0$ the map $\Psi$ reduces to a map $\Psi:C_d^{(a)}\to Gr(e+1,S^dU)=Gr(e,\PP^d)$ defined by $\Psi(D)=[\langle D\rangle]$, with $\langle D\rangle$ the linear span of $D$ considered as a subscheme of $C_d$ and hence of $\PP^d$. We already know this map is injective and hence birational to the image. Assume now that $r\geq 1$. 
We will prove the rationality of $G_\tau$ by induction on $B_\tau=\sum b_i$. If $B_\tau=0$ then $b_i=0$ for all $i=1,\ldots,r$, hence the map $\Psi$ becomes $\Psi(D,[f_1],\ldots,[f_r])=[\langle D\rangle\oplus \langle [f_1],\ldots [f_r]\rangle]$. The $\langle D\rangle$ summand is uniquely determined by $[T]= \Psi(D,f_1,\ldots,f_r)$, since $\langle D\rangle=S_T$ in the notations of Theorem \ref{thm:main1}.  Let us denote $G_a\subset Gr(a+1,S^dU)$ the variety of spaces $[S]=[\langle D\rangle]$ with $D\in C_d^a$ (this  part does not exist if $a=-1$). We have already proved that $G_a$ is rational.  Given a fixed $[T]=\Psi(D,[f_1],\ldots,[f_r])$. Let us take a complement $W$ of $\langle D\rangle$ in $S^dU$. Then we can parametrize an open set of $G_\tau$ by an open set in $C_d^{(a)}\times Gr(r,W)$. This shows that $G_\tau$ is rational for $\tau=(a,0,\ldots,0)$. Now assume the rationality of $G_{\tau'}$ has been proved for any $\tau$ such that $B_\tau'< B_\tau$ and let us examine $G_\tau$. We write $\tau=(a,b_1,\ldots,b_s,0,\ldots0)$ with $b_i>0$ for all $i=1,\ldots s$. Let us fix a space $[T]\in G_\tau$, so that $T=S_T\oplus \bigoplus_{i=1}^s\partial^{b_i}(f_i)\oplus T'$ with $T'$ a $(r-s)$ dimensional subspace of $S^dU$.
We have $\partial^{-1}T=S_T\oplus \bigoplus_{i=1}^s\partial^{b_i-1}(f_i)$ and $\partial(\partial^{-1}T)=S_T\oplus \bigoplus_{i=1}^s\partial^{b_i}(f_i)$. Note that the numerical type of $\partial^{-1}(T)$ is $\tau'=(a,b_1-1,\ldots,b_s-1)$ and that the summand $\partial(\partial^{-1}(T))$ of $T$ is uniquely determined by $T$. 
Taking a complement $W$ of $S_T\oplus \bigoplus_{i=1}^s\partial^{b_i}(f_i)$ in $S^dU$, we see that an open set in $G_\tau$ is parameterized by $G_{\tau'}\times Gr(r-s, W)$. Hence $G_\tau$ is rational.
\end{proof}
To illustrate the formula obtained for $\dim G_\tau$ in a relevant case, we apply it to the following proposition.
\begin{prop}\label{prop:generic1} Assume that $T$ is of type $\tau=(a,b_1,\ldots,b_r)=(-1,0,\ldots,0)$. In particular $\dim T=r$. Such a space $T$ exists only if $2r\leq d$. Then $G_\tau$ is an open set of the Grassmannian $Gr(r,S^dU)$, indeed one has $\dim G_\tau=r(d+1-r)$. \end{prop}
\begin{proof} Assume that $T$ has type as above. Then one has $\dim \partial T=2\dim T=2r\leq \dim S^{d-1}U=d$. Using formula (\ref{eq:dimGtau}) we find $\dim G_\tau=r+rd-r^2=r(d+1-r).$ \end{proof}
\section{Digression: computation of the generic $\partial$-type.} By the {\em generic $\partial$-type} we mean the numerical type of $T$ realized for $[T]$ varying in a dense open subset of $Gr(e+1,S^dU)$. Of course there is only one such type and we are going to determine it as a function of $d$ and $\dim T$. The case of $\dim T=d+1$ is trivial, the Grassmannian reduces to one point, and we can conventionally associate to it the type $(a)=(d)$, since in this case $T=\langle C_d\rangle$ and it is generated by $d+1$ points in the rational normal curve. Similarly, in the case $\dim T=d$ we have only  type $(a)=(d-1)$. Next we will assume $\dim T=e+1\leq d-1$. Then a general $T$ of such dimension does not intersect $C_d$ and hence it will be of type $(-1,b_1,\ldots,b_r)$, i.e. with $a=-1$. We distinguish two cases.
\vskip2mm
\paragraph{\bf Case 1.} We assume $2(e+1)=2\dim T\leq d$. Then the general $T=\langle f_1,\ldots,f_{e+1}\rangle$ has the property that $\partial(T)=\partial(f_1)\oplus\cdots\partial(f_{e+1})$ and $\dim\partial(T)=2(e+1)$. This can be easily seen finding a special $T$ with the property above, for example $T=\langle x^{d-1}y,x^{d-3}y^3,\ldots,x^{d-1-2e}y^{2e+1}\rangle$, so $\dim \partial(T)=2(e+1)$ also for the general $T$ by semicontinuity. This implies that the generic type for $2\dim T\leq d$ is $$(a,b_1,\ldots,b_{e+1})=(-1,0,\ldots,0).$$ The same result follows applying Proposition \ref{prop:generic1}.
\vskip2mm
\paragraph{\bf Case 2.} We now assume $2(e+1)>d>e+1$. We denote $r=d-(e+1)$ and $s=2(e+1)-d$. We claim that for $T$ general we have $\partial(T)=S^{d-1}U$ and therefore $\dim\partial(T)=d=e+1+r=\dim T+r$. Any $T$ of type $(-1,b_1,\ldots,b_r)$ with $\sum b_i=e+1-r=2(e+1)-d=s$ will have $\dim \partial(T)=d$ and hence $\partial(T)=S^{d-1}U$. By semicontinuity this holds also for the general $[T]\in Gr(e+1, S^dU)$, as we claimed. To find the generic type we impose the condition $\dim G_\tau=\dim Gr(e+1,S^dU)=(e+1)(d-e)$ and use formula (\ref{eq:dimGtau}) for $\dim G_\tau$, namely we impose
\begin{eqnarray*}(e+1)(d-e)&=&\dim G_\tau\\ &=&e+1+r(d-a-1)-\sum_{i=1}^r\sum_{b_j\geq b_i}(b_j-b_i+1))\\
&=&rd+e+1-\sum_{i=1}^r\sum_{b_j\geq b_i}(b_j-b_i+1)\\
&=&(d-e-1)d+e+1-\sum_{i=1}^r\sum_{b_j\geq b_i}(b_j-b_i+1).\end{eqnarray*} 
So we want \begin{eqnarray*}\sum_{i=1}^r\sum_{b_j\geq b_i}(b_j-b_i+1)&=&(d-e-1)d+e+1-(e+1)(d-e)\\
&=&(d-e-1)^2=r^2. \end{eqnarray*}
Note that from the topological considerations at the beginning of this subsection, we already know that there must exist only one $(b_1,\ldots,b_r)$ satisfying the equalities above. Write $r=h+k$ and $s=h(a+1)+ka=ra+k$, with $a,h,k\geq 0$. Then we set $(b_1,\ldots,b_h,b_{h+1},\ldots,b_{h+k})=(a+1,\ldots,a+1,a,\ldots,a)$, i.e. with $h$ entries equal to $a+1$ and $k$ entries equal to $a$. We have of course $\sum b_i=s$ and moreover $\sum_{i=1}^r\sum_{b_j\geq b_i}(b_j-b_i+1)=h^2+2hk+k^2=r^2$, so the generic type when $2(e+1)>d>e+1$ is the following:
$$(-1,b_1,\ldots,b_h,b_{h+1},\ldots,b_{h+k})=(-1,a+1,\ldots,a+1,a,\ldots,a),$$
with $h+k=d-(e+1)$ and $h(a+1)+ka=2(e+1)-d$.  
\section{Restricted tangent bundles to projective rational curves}
Any non degenerate rational curve $C\subset\PP^{d-e-1}$ of degree $d$ and with $e\geq 0$ can be identified, up a projectivity of $\PP^{d-e-1}$, with the image of the rational normal curve $C_d\subset\PP(S^dU)=\PP^d$ by means of a projection $\pi:\PP^d\dashrightarrow\PP^{d-e-1}$,  with vertex a space $\PP(T)\subset \PP^d$ with $\dim T=e+1$. This point of view was already adopted in \cite{Bernardi} to study the Hilbert schemes of rational curves with a given normal bundle or restricted tangent bundle. We can assume that the vertex $\PP(T)$ does not intersect $C_d$ and we get a parametrization $f:\PP^1\to C$ given as the composition $\PP^1\stackrel{\nu_d}\to C_d\stackrel{\pi}\to C$. The {\em restricted tangent bundle} is by definition $f^\ast T_{\PP^{d-e-1}}$ and, as a locally free sheaf on $\PP^1$, it is isomorphic to $\OO_{\PP^1}(a_1)\oplus\cdots\oplus \OO_{\PP^1}(a_{d-e-1})$, for suitable numbers $a_1\geq a_2\geq\cdots\geq a_{d-e-1}$.
The restricted Euler exact sequence \begin{equation}\label{eq:EulerC} 0\to \OO_{\PP^1}\to (S^dU/T)\otimes\OO_{\PP^1}(d)\to f^\ast(T_{\PP^{d-e-1}})\to 0,\end{equation}  and the fact that $C\subset\PP^{d-e-1}$ is non-degenerate show that \begin{equation}\label{eq:splitconstraints}a_1\geq\cdots\geq a_{d-e-1}\geq d+1,\quad\sum a_i=(d-e)d.\end{equation} 
In the remainder of this section we will name $s=d-e-1$.
\vskip2mm
\paragraph{\bf Remark} As it was already observed in \cite{Ran}, the exact sequence (\ref{eq:EulerC}) is determined by an extension class $\xi\in Ext^1(\bigoplus_{i=1}^s\OO_{\PP^1}(a_i),\OO_{\PP^1})$, and one knows that a general such extension has the form $0\to\OO_{\PP^1}\to\mathcal{E}\to \bigoplus_{i=1}^s\OO_{\PP^1}(a_i)\to 0$, with $\mathcal{E}\cong \OO_{\PP^1}^{s+1}(d)$. Note that we may recover the map $f$ from the sequence (\ref{eq:EulerC}). If $f:\PP^1\to\PP^s$ is given by $f=(f_0,\ldots,f_s)$ with each $f_i\in S^dU^\ast$, then the map $\OO_{\PP^1}\to\mathcal{E}$ in (\ref{eq:EulerC}) is given by $(f_0,\ldots,f_s)^t$.  This may be used to show the irreducibility and the rationality of the spaces of maps $f:\PP^1\to\PP^s$ with $\T_f=\bigoplus_{i=1}^s\OO_{\PP^1}(a_i)$. 

Our next results will show how to explicitly construct such maps and such rational curves as projections of the degree $d$ rational normal curve, explaining how to choose the vertices $\PP(T)$ of the needed projections.
\vskip2mm
\paragraph{}
For any $d\geq 0$ and $T\subset S^dU$ we denote $\langle\cdot,\cdot\rangle: S^dU^\ast\times S^dU\to\C$ the natural pairing and
 $$T^\perp=\{f\in S^dU^\ast\ |\ \langle f,\tau\rangle=0,\ \forall \tau\in T\}.$$
 Let us denote $T^\perp S^kU^\ast$ the image of the multiplication map $T\otimes S^kU^\ast\to S^{k+d}U^\ast$.
 
 Note that $S^dU/T\cong (T^\perp)^\ast$. Writing  $f^\ast(T_{\PP^{d-e-1}})=\T$, the restricted Euler sequence may also be written
 \begin{equation}\label{eq:EulerC1} 0\to \OO_{\PP^1}\to (T^\perp)^\ast\otimes\OO_{\PP^1}(d)\to \T\to 0.\end{equation}

The object of this section is to relate the {\em splitting type}  of $(a_1,\ldots,a_{d-e-1})$ of $\T$ and the {\em numerical type} $(-1,b_1,\ldots,b_r)$ of the vertex $T\subseteq S^dU$ introduced and studied in the preceding sections.  The following result will be useful for our purposes. 
\begin{lm}\label{lm:relTTast} For any $k\geq 0$ one has $(\partial^{-k}T)^\perp= T^\perp S^kU^\ast$. \end{lm}
\begin{proof} It is sufficient to show the statement in the case $k=1$, that is the equality $(\partial^{-1}T)^\perp=T^\perp U^\ast$. Indeed the general case will follow by recursion from the $k=1$ case, since one has $T^\perp S^{k+1}U^\ast=(T^\perp S^k U^\ast)U^\ast$ and $\partial^{-k-1}T=\partial^{-1}(\partial^{-k}T)$.

If $f\in T^\perp$, $l\in U^\ast$ and $g\in \partial^{-1}T$, one has $$\langle fl,g\rangle=\langle f,D_l(g)\rangle =0,$$
where we have denoted $D_l=\alpha\partial_x+\beta\partial_y$ the derivation associated to $l\in U^\ast$. This calculation shows that $T^\perp U^\ast\subseteq (\partial^{-1}T)^\perp$.

Now we have to show that $(\partial^{-1}T)^\perp\subseteq T^\perp U^\ast$. This is equivalent to show that $(T^\perp U^\ast)^\perp\subseteq \partial^{-1}T$. Let $\tau\in (T^\perp U^\ast)^\perp$, so that for any $f\in T^\perp$ one has $\langle f\partial_x, \tau\rangle=\langle f\partial_y, \tau\rangle=0$. Then we have $$\forall\ f\in T^\perp\ \langle f,\partial_x(\tau)\rangle=\langle f,\partial_y(\tau)\rangle=0,$$ which implies $\partial_x(\tau)\in T$ and $\partial_y(\tau)\in T$ and therefore $\tau\in\partial^{-1}T$. \end{proof}
\subsection{Cohomology computations on the restricted tangent sheaf}
Tensoring (\ref{eq:EulerC1}) by $\OO(-k)$ we obtain the following exact sequences of vector spaces:
\begin{equation}\label{eq:cohomRT0} 0\to (T^\perp)^\ast\otimes H^0\OO(d-k)\to H^0\T(-k)\to H^1\OO(-k)\to 0 ,\end{equation} 
for any $1\leq k\leq d+1$ and
\begin{equation}\label{eq:cohomRT1} 0\to H^0\T(-k)\to H^1\OO_{\PP^1}(-k)\stackrel{\rho}\to 
(T^\perp)^\ast\otimes H^1\OO_{\PP^1}(d-k)\to H^1\T(-k)\to 0,\end{equation}
for any $k\geq d+2$.
Note that for $k\geq d+2$ the map $\rho$ is the dual of the multiplication map 
\begin{equation*}\label{eq:Serredualrho} T^\perp\otimes S^{k-d-2}U^\ast\to S^{k-2}U^\ast.\end{equation*}
Combining this observation with the result of Lemma \ref{lm:relTTast}, we get the following formulas
\begin{eqnarray*} h^0 \T(-k)&=&k-1-\dim (T^\perp S^{k-d-2}U^\ast)\\
&=&k-1-\dim (\partial^{-k+d+2}T)^\perp\\
&=&\dim (\partial^{-k+d+2}T).\end{eqnarray*}
The formula obtained above holds for any $k\geq d+2$. Note that for $k=d+2$ one always has $h^0 \T(-d-2)=\dim T$. On the other hand we have 
$H^0\T(-d-1)\cong H^1\OO_{\PP^1}(-d-1)=S^{d-1}U^\ast$ and from (\ref{eq:cohomRT0}) and  Serre duality we get $H^0\T(-k)\cong ((T^\perp)^\ast\otimes S^{d-k}U^\ast)\oplus (S^{k-2}U^\ast)^\ast$ for $1\leq k\leq d$.
\subsection{Splitting type of $\T$}
Once the dimensions of the spaces of global sections $H^0\T(-k)$ are known for any $k\in \Z$, some standard calculations allow to compute the splitting type $\T=\OO_{\PP^1}(a_1)\oplus\cdots\oplus \OO_{\PP^1}(a_{s})$. The object of this section is to relate it to the numerical type of $T$.
\vskip2mm
\paragraph{\bf Example.} Let us set $d=5$ and $e=1$, corresponding to considering rational quintic curves $C\subset\PP^3$. The vertex $\PP(T)\cong\PP^1$ of the projection $\PP^5\dashrightarrow\PP^3$ has numerical type either $(-1,0,0)$ or $(-1,1)$. On the other hand, the two possible splitting types for $\T$ are $(8,6,6)$ or $(7,7,6)$, obtained from the conditions (\ref{eq:splitconstraints}). 

If $\T=\OO(8)\oplus\OO(6)^2$ then we have
$h^0\T(-7)=2=\dim T$ and $h^0\T(-8)=1=\dim \partial^{-1}T$. This implies that the numerical type of $T$ corresponding to $\T=\OO(8)\oplus\OO(6)^2$ is $(-1,1)$. Then the numerical type corresponding to $\T=\OO(7)^2\oplus\OO(6)$ must be $(-1,0,0)$ and indeed one verifies that in this case $h^0\T(-8)=0=\dim\partial^{-1}T$.
\vskip2mm
\paragraph{\bf Calculations in the general case}
As above we write $$\T\cong \OO(a_1)\oplus\cdots\oplus\OO(a_{s})$$ with $a_1\geq\cdots\geq a_{s}\geq d+1$ and $\sum a_i=(d-e)d$.  We have seen that $h^0\T(-k)=\dim\partial^{-k+d+2}T$ for any $k\geq d+2$ and $h^0\T(-d-1)=d$. Let us set $$h=k-d-2,\quad c_i=a_i-d-2\ \mbox{ for }\ i=1,\ldots,d-e-1.$$ Then from (\ref{eq:splitconstraints}) we get $c_1\geq\cdots\geq c_{d-e-1}\geq -1$ and  $\sum (c_i+1)=e+1$.

This is compatible with $\T(-d-1)=\OO(c_1+1)\oplus\cdots\oplus\OO(c_{d-e-1}+1)$ and $\sum_{i=1}^{d-e-1} (c_i+2)=h^0\T(-d-1)=d$. We find also
$$\dim \partial^{-h}T=h^0\T(-d-2-h)=\sum_{c_i\geq h}(c_i-h+1)$$ for all $h\geq 0$. Taking the numerical type $(-1,b_1,\ldots,b_r)$ of $T$ into account, we have 
$$\dim \partial^{-h}T=\sum_{b_j\geq h}(b_j-h+1).$$
We have $d-e-1=d-\sum (b_j+1)=s$. Recall that the inclusion $\partial T\subseteq S^{d-1}U$ implies $\dim \partial T=\dim T+r\leq d-1$, so that $r\leq d-e-2<s$. 

We see that $b_1\geq\cdots \geq b_r\geq 0$ and $c_1\geq \cdots\geq c_s\geq -1$ are two sequences with $r<s$ and we denote
\begin{equation}\label{eq:relbc}
S(h)=\sum_{b_j\geq h}(b_j+1-h)=\sum_{c_i\geq h}(c_i+1-h),\ \forall\ h\geq 0.\end{equation}
Let us assume that $c_1\geq\cdots\geq c_{r'}\geq 0$ and $c_{r'+1}=\cdots=c_s=-1$. Then the conditions (\ref{eq:relbc}) involve only the $c_i$'s with $1\leq i\leq r'$. Taking the differences $\Delta S(h)=S(h)-S(h+1)$ for any $h\geq 0$ one finds $\Delta S(h)=|\{j\ |\ b_j\geq h\}|=|\{i\ |\ c_i\geq h\}|$ for any $h$, and taking once more the differences $\Delta^2 S(h)=\Delta S(h)-\Delta S(h+1)$ one finds $$|\{j\ |\ b_j= h\}|=|\{i\ |\ c_i= h\}|\quad \forall\ h\geq 0.$$ This implies 
$r'=r$ and $(b_1,\ldots,b_r)=(c_1,\ldots,c_r).$ Then we get $$(c_1,\ldots,c_r,c_{r+1},\ldots,c_{s})=(b_1,\ldots,b_r,-1,\ldots,-1).$$
We have finally the formula for the splitting type of $\T$
\begin{equation}\label{eq:tangent} 
\T=\OO(b_1+d+2)\oplus \cdots\OO(b_r+d+2)\oplus\OO(d+1)^{s-r}.
\end{equation}
We summarize the results of this section by stating the last theorem of this article. As in the beginning of this section, let $C$ be non degenerate degree $d$ rational curve in $\PP^{s}$ and let us recall that $C$ is parametrized by $f:\PP^1\to \PP^{s}$ with $f=\pi\circ \nu_d$ where $\pi:\PP(S^dU)\dashrightarrow \PP^{s} $ is a projection with vertex $\PP(T)$ of dimension $e$. Then one has the following result. 
\begin{thm}\label{thm:main3} The following facts are equivalent.
\begin{enumerate}
\item The numerical type of $T$ is  $(-1,b_1,\ldots,b_r)$.
\item The restricted tangent bundle $\T=f^\ast T_{\PP^{s}}$ has splitting type (\ref{eq:tangent}).
\end{enumerate}
Moreover the variety of maps $f:\PP^1\to \PP^{s}$ parametrizing rational curves as above and with $\T$ of the form (\ref{eq:tangent}) is irreducible, smooth and rational.
\end{thm}
\begin{proof} We only add some remarks on the smoothness, irreducibility and  rationality of the varieties of maps as above. This is a consequence of the parametrization of maps $\PP^1\to\PP^s$ by elements of a suitable Ext group, and a complete proof of this fact can be found in \cite{Ramella}. However we will give an independent proof of the irreducibility and rationality of these varieties that arises from the explicit parametrization of them. 

Fixing a map $f:\PP^1\to \PP^{s}$ amounts to fixing a subspace $T'\subset S^dU^\ast$ with $\dim T'=d-e$ and a basis $(f_0,\ldots,f_{s})$ of it. Fixing $T'$ is equivalent to fixing the $e+1$ dimensional annihilator $T=(T')^\perp\subset S^dU$ of $T'$. The condition that $\T$ has the form (\ref{eq:tangent}) has been shown to be equivalent to $T$ being of type $\tau=(-1,b_1,\ldots,b_r)$. Hence the space of maps under consideration is a principal $PGL(d-e)$ bundle on $G_\tau$,  which is a irreducible rational variety since $PGL(d-e)$ and $G_\tau$ are irreducible and rational, the second one by Theorem \ref{thm:main2}.  \end{proof}
\section{Concluding remarks}
From the results of the previous section we obtain the following procedure for constructing all parametrized rational curves $\PP^1\stackrel{f}\to C\subset\PP^s$ of degree $d$, with $s=d-e-1<d$ and restricted tangent bundle $\T_f=\bigoplus \OO_{\PP^1}(a_i)$, with $a_1\geq a_2\geq\cdots\geq a_s\geq d+1$ and $\sum a_i=(d-e)d$. Set $b_i=a_i-d-2$ for any $i$ such that $a_i\geq d+2$. In this way we obtain $b_1\geq \cdots\geq b_r\geq 0$ with $r=\max\{i\ |\ a_i\geq d+2\}$. Take any forms $f_i\in S^{d+b_i}U$, with $[f_i]\not\in Sec^{b_i}C_{d+b_i}$ for any $i=1,\ldots,r$ and  such that the sum $\partial^{b_1+1}(f_1)+\cdots+\partial^{b_r+1}(f_r)$ is direct in $S^{d-1}U$. This is equivalent to imposing that the dimension of this latter space is the maximum possible, i.e. equal to $\sum_{i=1}^r(b_i+2)$ and it amounts to impose that the matrix of coefficients of the $\partial_x^{b_i-j}\partial_y^jf_i$'s with respect to of the standard monomial basis $x^iy^j$ avoids a well defined determinantal locus. Then one sets $$T=\partial^{b_1}(f_1)+\cdots+\partial^{b_r}(f_r)\subset S^dU.$$ By the results of the preceding sections one sees that $T$ has type $(-1,b_1,\ldots,b_r)$. Now consider the curve $C\subset\PP^s=\PP(S^dU/T)$ obtained as projection of the rational normal curve $C_d=\nu_d(\PP^1)\subset \PP^d=\PP(S^dU)$ from the vertex $\PP(T)$. Then $C$ has restricted tangent bundle isomorphic to $\bigoplus_{i=1}^s \OO_{\PP^1}(a_i)$.
Moreover any other rational curve with this property is obtained from such a $C$ after a projective transformation of $\PP^s$.
\subsection{Future developments.} In this article we did not touch the problem of studying the Hilbert schemes of rational curves $C\subset \PP^s$ with a given decomposition of the normal bundle $\N_C$ as a vector bundle on $\PP^1$, see for example \cite{Bernardi} and \cite{Ran} for some results on the problem in a general projective space $\PP^s$ and \cite{Ghione}, \cite{Sacchiero} and \cite{Eisen} for the case of rational curves in $\PP^3$. 

We defer a detailed study of the above problem, by using the techniques of this article, to a future paper in preparation.
\subsection*{Acknowledgements} We wish to express our gratitude  to Giorgio Ottaviani and Francesco Russo for the many pleasant and stimulating discussions during the development of this work. We also thank an anonimous referee for a very careful revision of the paper.

\end{document}